\documentclass{amsart}
              [1997/02/02 v1.2e PROC Author Class]

\copyrightinfo{2006}
  {American Mathematical Society}

\newtheorem{theorem}{Theorem}[section]
\newtheorem{corollary}[theorem]{Corollary}

\newtheorem{proposition}[theorem]{Proposition}
\newtheorem{example}[theorem]{Example}

\begin{document}

\title{Linearly ordered compacts and co-Namioka spaces}

\author{V.V.Mykhaylyuk}
\address{Department of Mathematics\\
Chernivtsi National University\\ str. Kotsjubyn'skogo 2,
Chernivtsi, 58012 Ukraine}
\email{vmykhaylyuk@ukr.net}



\commby{Ronald A. Fintushel}



\begin{abstract}
It is shown that for any Baire space $X$, linearly ordered compact $Y$ and separately continuous mapping $f:X\times Y\to\mathbb R$ there exists a dense in $X$ $G_\delta$-set $A\subseteq X$ such that $f$ is jointly  continuous at every point of $A\times Y$, i.e. any linearly ordered compact is a co-Namioka space.
\end{abstract}

\maketitle
\section{Introduction}

Investigation of the joint continuity points set of separately continuous functions defined on the product of a Baire space and a compact space take an important place in the separately continuous mappings theory. The classical Namioka's result [1] become the impulse to the intensification of these investigations. In particular, the following notions were introduced in [2].

Let $X$, $Y$ be topological spaces. We say that a separately continuous function $f:X\times Y\to \mathbb R$ {\it has the Namioka property} if there exists a dense in $X$ $G_\delta$-set $A\subseteq X$ such that $f$ is jointly continuous at every point of set $A\times Y$.

A compact space $Y$ is called {\it a co-Namioka space} if for every Baire space $X$ each separately continuous napping $f:X\times Y\to \mathbb R$ has the Namioka property.

Most general results in the direction of study of co-Namioka space properties were obtained in [3,4]. It was obtained in [3.4] that the class of compact co-Namioka spaces is closed over products and contains all Valdivia compacts. Moreover, it was shown in [3] that the linearly ordered compact $[0,1]\times \{0,1\}$ with the lexicographical order is co-Namioka and it was reproved in [3]  (result from [5]) that every completely ordered compact is co-Namioka. Thus, the following question naturally arises: is every linearly ordered compact co-Namioka?

In this paper we give the positive answer to this question.

\section{Definitions and auxiliary statements}

Let $X$ be a topological space and $f:X\to \mathbb R$. For every nonempty set $A\subseteq X$ by $\omega_f(A)$ we denote the oscillation
$${\rm sup}\{|f(x')-f(x'')|:x',x''\in A\}$$
of the function $f$ on the set $A$. Moreover for every point $x_o\in X$ by $\omega_f(x_o)$ we denote the oscillation ${\rm inf}\{\omega_f(U):U\in {\mathcal U}\}$ of the function $f$ at $x_o$, where $\mathcal U$ is the system of all neighborhoods of $x_o$ in $X$.

Let $X$, $Y$ be topological spaces, $f:X\times Y\to \mathbb R$, $x_o\in X$ and $y_o\in Y$. The mappings $f^{x_o}$ and $f_{y_o}$ are defined by: $f^{x_o}(y)=f_{y_o}(x)=f(x,y)$ for every $x\in X$ ³ $y\in Y$.

For a linearly ordered set $(X,<)$ and a point $x', x'' \in X$, $x'<x''$, we put 
$$[x',x'']=\{x\in X: x'\leq x\leq x''\},\,\,\,[x',x'')=\{x\in X: x'\leq x< x''\},$$
$$(x',x'']=\{x\in X: x'< x\leq x''\}\,\,\,{\rm and}\,\,\,(x',x'')=\{x\in X: x'< x < x''\}.$$ Points $x', x''\in X$, $x'<x''$ is called {\it neighbor}, if $(x',x'')=\O$. Recall that all nonempty open intervals $(x',x'')$ and intervals $[a,x)$ and $(x,b]$, where $a$ and $b$ are the minimal and maximal elements in $X$ respectively (if they exist) form a base of the topology on $X$. It easy to see that a linearly ordered space $X$ is a compact space with respect to the topology generated by the linear order, if and only if every nonempty closed set $A\subseteq X$ has in $X$ minimal and maximal elements.

A topological space $X$ is called {\it connected}, if $A\cup B\ne X$ for every disjoint nonempty open in $X$ sets $A$ and $B$. Note that a linearly ordered compact $X$ is connected if and only if $X$ has not neighbor elements.

\begin{proposition} \label{p:2.1} Let $(X,<)$ be a linearly ordered compact, $f:X\to\mathbb R$ be a continuous function and $\varepsilon>0$. Then there exists
$n\in\mathbb N$ such that for every elements $a_1, a_2, \dots , a_n, b_1, b_2, \dots , b_n \in X$ such that $a_1< b_1\leq a_2 < b_2 \leq \dots \leq a_n<b_n$, there exists $k\leq n$ such that $|f(a_k)-f(b_k)|<\varepsilon$.
\end{proposition}

\begin{proof} Consider a finite cover $(U_i:i\in I)$ of compact space $X$ by intervals $U_i$ such that $\omega_f(U_i)<\varepsilon$ for every $i\in I$. Put $n=|I|+1$. Then for every points $a_1< b_1\leq a_2 < b_2 \leq \dots \leq a_n<b_n$ from the space $X$ there exists $i_o\in I$ such that the interval $U_{i_o}$ coincides at least three of these points. Taking into account that $U_{i_o}$ is an interval we obtain that there exists $k\leq n$ such that $a_k, b_k \in U_{i_o}$. Then $$|f(a_k)-f(b_k)|\leq \omega_f(U_{i_o})<\varepsilon.$$
\end{proof}

\begin{proposition}\label{p:2.2} Let $(X,<)$ be a linearly ordered connected compact, $a=\min X$, $b=\max X$, $f:X\to {\mathbb R}$ be a continuous mappings and  $f(a)\ne f(b)$. Then there exists a point $c\in (a,b)$ such that $f(c)=\frac{1}{2}(f(a)+f(b))$.\end{proposition}

\begin{proof} We put $y_o=\frac{1}{2}(f(a)+f(b))$, $A=f^{-1}((-\infty,y_o)$ and $B=f^{-1}((y_o,+\infty))$. Since $f$ is continuous, the sets $A$ and $B$ are open in $X$. It follows from the connectivity of $X$ that the set $C=X\setminus (A\cup B)$ is nonempty. It remains to chose a point $c\in C$.
 \end{proof}

\section{Main results}

\begin{theorem}\label{th:3.1} Every linearly connected ordered compact is co-Namioka space.
\end{theorem}

\begin{proof} Let $X$ be a Baire space, $(Y,<)$ be a linearly ordered connected compact and $f:X\times Y\to\mathbb R$ be a separately continuous function. We prove that $f$ has Namiola property.

Fix $\varepsilon >0$ and show that the open set 
$$G_{\varepsilon}=\{x\in X: \omega_f(x,y)< 9 \varepsilon \mbox{\,\,for\,\,every}\,\,y\in Y\}$$
is dense in $X$.

Let $U$ be a nonempty open in $X$ set. For every $x\in U$ we denote by $N_x$ the set of all $n\in {\mathbb N}$ such that there exist  $a_1,\dots,a_n,b_1,\dots,b_n\in Y$ such that $$a_1<b_1\le a_2<b_2\le \dots\le a_n<b_n\,\,\,{\rm and}\,\,\,|f(x,a_i)-f(x,b_i)|>\varepsilon$$ for every $i=1,\dots,n$. It follows from Proposition 2.1 that all sets $N_x$ are upper bounded. For every $x\in U$ we put $\varphi(x)=\max N_x$, if $N_x$ is nonempty, and $\varphi(x)=0$, if $N_x=\O$. It follows from the continuity of $f$ with respect to the first variable that for every $n$ the set $\{x\in U:\varphi(x)>n\}$ is open in $U$, i.e. the function $\varphi:U\to {\mathbb Z}$ is upper semicontinuous on the Baire space $U$. Therefore (see [6]) the function $\varphi$ is pointwise discontinuous. Then there exist an open in $U$ nonempty set $U_o$ and nonnegative $n\in {\mathbb Z}$ such that $\varphi(x)=n$ for every $x\in U_o$.

If $n=0$, then $|f(x,a)-f(x,b)|\leq \varepsilon$ for every $x\in U_o$ ³ $a,b\in Y$. Then taking any point $y_o\in Y$ and an open nonempty set $U_1\subseteq U_o$ such that $\omega_{f_{y_o}}(U_1)<\varepsilon$, we obtain that $\omega_f(x,y)<3 \varepsilon$ for every $x\in U_1$ ³ $y\in Y$. In particular, $U_1\subseteq G_\varepsilon$.

Now we consider the case of $n\in \mathbb N$. We take a point $x_o\in U_o$ and choose $a_1,\dots,a_n,b_1,\dots,b_n\in Y$ such that $a_1<b_1\le a_2<b_2\le
\dots\le a_n<b_n$ and  $|f(x_o,a_i)-f(x_o,b_i)|>\varepsilon$ for $1\leq i\leq n$.

Using the continuity of $f$ with respect to the first variable we choose an open neighborhood $U_1\subseteq U_o$ of $x_o$ in $U$ such that $|f(x,a_i)-f(x,b_i)|>\varepsilon$ for every $x\in U_1$ and $i\in \{1,\dots, n\}$. Show that for every $y_o\in Y$ there exists an open neighborhood $V$ of $y_o$ in $Y$ such that $\omega_{f^x}(V)\leq 4\varepsilon$ for every $x\in U_1$.

Let $y_o\in G=Y\setminus \bigcup\limits^{n}_{i=1}[a_i,b_i]$. Since the set $G$ is open in $Y$, there exists an open in $Y$ interval $V$ such that $V\subseteq G$. Then for every $a,b\in V$ with  $a<b$we have $[a,b]\cap[a_i,b_i]=\O$ for every $i=1,\dots, n$. Taking into account that $\varphi(x)=n$ and $|f(x,a_i)-f(x,b_i)|>\varepsilon$ for every $x\in U_1$ and $i\in \{1,\dots, n\}$ we obtain that $|f(x,a)-f(x,b)|\leq \varepsilon$, i.e. $\omega_{f^x}(V)\leq \varepsilon$ for every $x\in U_1$.

Let $a_i<y_o<b_i$ for some $i\in \{1,\dots,n\}$. Then we put $V=(a_i,b_i)$. Note that $|f(x,a)-f(x,b)|\leq 2\varepsilon$ for every points $a,b \in (a_i,b_i)$ and $x\in U_1$. Really, we suppose that there exist $a,b \in (a_i,b_i)$ and $x\in U_1$ such that  $|f(x,a)-f(x,b)|> 2\varepsilon$. Then according to Proposition 2.2 there exist a point $c\in (a,b)$ such that $|f(x,a)-f(x,c)|=|f(x,c)-f(x,b)|>\varepsilon$. But this contradicts to $\varphi(x)=n$. Thus, $\omega_{f^x}(V)\leq 2\varepsilon$.

It remains to consider the case of $y_o\in \{a_i,b_i:1\leq i \leq n\}$. Let $a_o=\min Y$, $b_o=\max Y$ and $a_o<y_o<b_o$. Put $y_1=\max(\{a_i,b_i:0\leq i\leq n\}\cap [a_o, y_o))$, $y_2=\min(\{a_i,b_i:0\leq i\leq n\}\cap(y_o,b_o])$ and $V=(y_1,y_2)$. It follows from Proposition 2.2 that for every $a\in(y_1,y_o]$, $b\in[y_o,y_2)$ and $x\in U_1$ the following inequalities $|f(x,a)-f(x,y_o)|\leq 2\varepsilon$ and $|f(x,y_o)-f(x,b)|\leq 2\varepsilon$ hold. Therefore $\omega_{f^x}(V)\leq 4\varepsilon$ for every $x\in U_1$. In the case $y_o=a_o$ or $y_o=b_o$, it enough to put $V=[y_o,y_2)$ or $V=(y_1,y_o]$.

Now we prove that $U_1\subseteq G_{\varepsilon}$. Let $(x_o,y_o)\in U_1\times Y$. Take an open neighborhood $V$ of $y_o$ in $Y$ such that $\omega_{f^x}(V)\leq 4\varepsilon$ for every $x\in U_1$. Using the continuity of $f$ with respect to the first variable we choose a neighborhood $U_2\subseteq U_1$ of $x_o$ in $X$
such that $\omega_{f_{y_o}}(U_2)<\varepsilon$. Then $\omega_f(U_2\times V)<9\varepsilon$.

Thus for every $\varepsilon$ we have that the set $G_{\varepsilon}$ is dense in Baire space $X$. Then the $G_{\delta}$-set $A=\bigcap\limits^{\infty}_{n=1} G_{\frac{1}{n}}$ is dense in $X$ too. Moreover, the function $f$ is jointly continuous at every point of set $A\times Y$, i.e. $f$ has Namioka property.
\end{proof}

\begin{corollary} \label{c:3.2} Every linearly ordered compact is a co-Namioka space.\end{corollary}

\begin{proof} Let  $(Y,<)$ be a linearly ordered compact. If the set $D$ of all pairs of neighbor points in $Y$ is empty, then according to Theorem 3.1 the space $Y$ is co=Namioka.

Let $D\ne \O$. For every pair $d\in D$ od neighbor points in $Y$ we denote by $a_d$ and $b_d$ the left and the right neighbor points of the pair $d$, i.e. $d=\{a_d,b_d\}$ and $a_d<b_d$. We put $Z=Y\cup(D\times (0,1))$ and define a linear order on $Z$, which is an extension of the order on $Y$. Let  $z'\in Y$, $d\in D$, $t\in (0,1)$ and $z''=(d,t)$. Then $z'<z''$, if $z'\leq a_d$, and $z''<z'$, if $b_d\leq z'$. Let $z'=(d',t'),z''=(d'',t'')\in D\times (0,1)$. Then $z'<z''$, if $a_{d'}<a_{d''}$, or $d'=d''$ and $t'<t''$.

Note that $(Z,<)$ is a compact space which has not neighbor elements. Moreover the space $(Y,<)$ is a compact subspace of $(Z,<)$.

Let $X$ be a Baire space and $f:X\times Y\to \mathbb R$ be a separately continuous function. We construct a mapping $g:X\times Z\to \mathbb R$, which is a extension of $f$. For every $x\in X$ and $z=(d,t)\in D\times (0,1) $ we put $g(x,z)=(1-t)f(x,a_d)+tf(x,b_d)$. It easy to see that the function $g$ is separately continuous too. Therefore according to Theorem 3.1 there exists a dense in $X$ $G_{\delta}$-set $A\subseteq X$ such that $g$ is jointly continuous at every point of set $A\times Z$. Therefore the function $f$ is jointly continuous at every point of the set $A\times Y$. Thus, $f$ has Namioka property and $Y$ is co-Namioka.
\end{proof}

\section{Example}

In this section we show that condition of continuity of $f$ with respect to the first variable in the above reasoning can not be replaces by quasicontinuity.

Let $X$ be a topological space and $f:X\to Y$. Recall that the mapping $f$ is called {\it quasicontinuous at a point $x_o\in X$}, if for every neighborhoods $U$ of $x_o$ in $X$ and $V$ of $y_o=f(x_o)$ in $Y$ there exists an open in $X$ nonempty set $G\subseteq U$ such that $f(G)\subseteq V$. A mapping $f$ is called {\it quasicontinuous}, if $f$ is quasicontinuous at every point $x\in X$.

\begin{example} Let $X=(0,1)$ and $Y=[0,1]\times \{0,1\}$ be a linearly ordered compact with the lexicographical order, i.e. $(y,i)<(z,j)$, if $y<z$ or $y=z$ and $i<j$. For every $t\in [0,1]$ we put $t_l=(t,0)$ and $t_r=(t,1)$. The function $f:X\times Y\to \mathbb R$ we define by: $f(x,y)=0$, if $x_r\leq y$, and $f(x,y)=1$, if $x_l\geq y$.

For every $x\in (0,1)$ we have $(f^x)^{-1}(0)=[x_r,1_r]$ and $(f^x)^{-1}(1)=[0_l,x_l]$. Therefore all functions $f^x$ are continuous. If $y\in \{0_l,0_r\}$, then $f_y(x)=1$ for every $x\in X$, and if $y\in \{1_l,1_r\}$, then $f_y(x)=0$ for every $x\in X$. Let $z\in (0,1)$. Then for $y=z_l$ we have $f_y(x)=0$, if $x\in
(0,z)$, and $f_y(x)=1$, if $x\in [z,1)$. And for $y=z_r$ we have $f_y(x)=0$, if $x\in (0,z]$, and $f_y(x)=1$, if $x\in (z,1)$. Thus, the function $f$ is quasicontinuous with respect to the first variable. But the function $f$ is jointly discontinuous at every point $(x,x_l)$ and $(x,x_r)$ for $x\in X$.\end{example}

\section{Acknowledgment}
The author would like to thank Maslyuchenko O.V. for his helpful comments.

\bibliographystyle{amsplain}

\begin{thebibliography}{10}

\bibitem {N} Namioka I. {\it Separate continuity and joint continuity} Pacif. J. Math. 51, N2 (1974), 515-531.

\bibitem {D} Debs G. {\it Points de continuite d'une function separement continue} Proc. Amer. Math. Soc.{\bf 97}, N 1 (1986), 167-176.

\bibitem {Bu} Bouziad A. {\it Notes sur la propriete de Namioka} Trans. Amer. Math. Soc. 344, N2 (1994), 873-883.

\bibitem {Bu1} Bouziad A. {\it The class of co-Namioka spaces is stable under product} Proc. Amer. Math. Soc. 124, N3 (1996) 983-986.

\bibitem {De} Deville R. {\it  Convergence ponctuelle et uniforme sur un espace compact} Bull. Acad. Polon. Sci. Ser. Math. 37 (1989) 507-515.

\bibitem {CT} Calbrix J., Troallic J.P. {\it Applications s\'{e}par\'{e}ment continues} C.R. Acad. Sc. Paris S\'{e}c. A. 288 (1979), 647 - 648.

\end{thebibliography}

\end{document}